\documentclass[lectures]{pcms-l}

\usepackage{amssymb}
\usepackage{graphics}
\usepackage{latexsym}
\usepackage{mathrsfs}
\usepackage{xspace}
\usepackage{enumerate}

\usepackage{hyperref}
\hypersetup{colorlinks,citecolor=blue}

\newtheorem{theorem}{Theorem}

\newtheorem{examps}[theorem]{Examples}
\newtheorem{exercise}[theorem]{Exercise}

\newcommand{\BZ}{{\mathbb{Z}}}
\newcommand{\BN}{{\mathbb{N}}}
\newcommand{\BR}{{\mathbb{R}}}
\newcommand{\BC}{{\mathbb{C}}}
\newcommand{\BF}{{\mathbb{F}}}
\newcommand{\BQ}{{\mathbb{Q}}}

\newcommand{\BG}{{\mathbb{G}}}

\newcommand{\BH}{{\mathbb{H}}}
\newcommand{\BK}{{\mathbb{K}}}

\newcommand{\OO}{{\mathcal{O}}}

\newcommand{\gD}{\Delta}
\newcommand{\gd}{\delta}
\newcommand{\gb}{\beta}

\newcommand{\gC}{\Gamma}
\newcommand{\gc}{\gamma}
\newcommand{\gs}{\sigma}
\newcommand{\gS}{\Sigma}
\newcommand{\gO}{\Omega}
\newcommand{\go}{\omega}
\newcommand{\gep}{\epsilon}

\newcommand{\ga}{\alpha}

\newcommand{\ti}[1]{\tilde{#1}}

\newcommand{\rank}{\text{rank}}
\newcommand{\Ad}{\text{Ad}}
\newcommand{\vol}{\text{vol}}

\newcommand{\SL}{\text{SL}}
\newcommand{\GL}{\text{GL}}
\newcommand{\PSL}{\text{PSL}}

\newcommand{\SO}{\text{SO}}

\newtheorem{prop}{Proposition}[section]
\newtheorem{thm}[prop]{Theorem}
\newtheorem{lem}[prop]{Lemma}
\newtheorem{cor}[prop]{Corollary}
\newtheorem{conj}[prop]{Conjecture}
\theoremstyle{definition}

\newtheorem{defn}[prop]{Definition}
\newtheorem{rem}[prop]{Remark}

\newtheorem{exam}[prop]{Example}

\begin{document}
\frontmatter
\tableofcontents
\mainmatter
\LectureSeries[Lattices]{Lectures on Lattices and locally symmetric spaces
\author{Tsachik Gelander}}
\address{Einstein Institute of Mathematics, The Hebrew University of Jerusalem
Jerusalem, 91904, Israel}



%

\section*{Introduction}
The aim of this short lecture series is to expose the students to the beautiful theory of lattices by, on one hand, demonstrating various basic ideas that appear in this theory and, on the other hand, formulating some of the celebrated results which, in particular, shows some connections to other fields of mathematics. The time restriction forces us to avoid many important parts of the theory, and the route we have chosen is naturally biased by the individual taste of the speaker.

\lecture{A brief overview on the theory of lattices}

The purpose of the first lecture is to introduce the student to the theory of lattices. While some parts of this lecture will be done consistently with full detail (some of which are given as exercises), other parts should be considered as a story aiming to give a broader view on the theory. 

\section
{Few definitions and examples}

Let $G$ be a locally compact group equipped with a left Haar measure $\mu$, i.e. a Borel regular measure which is finite on compact sets, positive on open sets and invariant under left multiplications (i.e. $\mu(gA)=\mu(A)$ for every $g\in G$ and measurable $A\subset G$) --- by Haar's theorem such $\mu$ exists and is unique up to normalization. The group $G$ is called unimodular if $\mu$ is also right invariant, or equivalently if it is symmetric in the sense that $\mu (A)=\mu (A^{-1})$ for every measurable set $A$. Note that $G$ is compact iff $\mu(G)<\infty$. 

For example:
\begin{itemize}
\item Discrete groups, abelian groups, compact groups and Perfect groups (i.e. groups that are equal to their commutator subgroup) are unimodular.
\item The group of affine transformations of the real line is not unimodular.
\end{itemize}

A closed subgroup $H\le G$ is said to be {\it co-finite} if the quotient space $G/H$ admits a non-trivial finite $G$ invariant Borel regular measure. 
Like in Haar's theorem, it can be shown that a $G$ invariant measure on $G/H$, if exists, is unique up to scaling (see \cite[Lemma 1.4]{Rag}). 
A {\it lattice} in $G$ is a co-finite discrete subgroup. A discrete subgroup $\gC\le G$ is a lattice iff it admits a finite measure fundamental domain, i.e. a measurable set $\gO$ of finite measure which forms a set of right coset representatives for $\gC$ in $G$ --- we will always normalize the measures so that  $\vol(G/\gC)=\mu(\gO)$. 
Let us denote $\gC\le_LG$ to express that $\gC$ is a lattice in $G$. 

\begin{exercise}\label{ex:1}
Show that every two fundamental domains have the same measure.
\end{exercise}

\begin{exercise}\label{ex:L=>um}
Deduce from Exercise \ref{ex:1} that if $G$ admits a lattice then it is unimodular. 
\end{exercise}

We shall say that a closed subgroup $H\le G$ is {\it uniform} if it is cocompact, i.e. if $G/H$ is compact. 

\begin{exercise}
A uniform discrete subgroup $\gC\le G$ is a lattice. Note that if $G=\SL_2(\BR)$ and $H$ is the Borel subgroup of upper triangular matrices, then $H$ is uniform but not co-finite.
\end{exercise}

\begin{examps}
\begin{enumerate}
\item If $G$ is compact every closed subgroup is cofinite. The lattices are the finite subgroups.
\item If $G$ is abelian, a closed subgroup $H\le G$ is cofinite iff it is uniform. 
\item Let $G$ be the Heisenberg group of $3\times 3$ upper triangular unipotent  matrices over $\BR$ and let $\gC=G(\BZ)=G\cap\SL_3(\BZ)$ be the integral points. Then $\gC$ is a uniform lattice in $G$. 
\item Let $T$ be a $k$ regular tree equipped with a $k$ coloring of the edges such that neighboring edges have different colors. Let $G=\text{Aut}(T)$ be the group of all automorphisms of $T$ considered with the open-compact topology and let $\gC$ be the group of those automorphisms that preserve the coloring. Then $\gC$ is a uniform lattice in $G$. (To study more about lattices in $\text{Aut}(T)$, see \cite{BaLu}.) 
\item $\SL_n(\BZ)$ is a non-uniform lattice in $\SL_n(\BR)$ (see \cite{Si} or \cite[Ch. X]{Rag}).
\item Let $\gS_g$ be a closed surface of genus $g\ge 2$. Equip $\gS_g$ with a hyperbolic structure and fix a base point and a unit tangent vector. The action of the fundamental group $\pi_1(\gS_g)$ via Deck transformations on the universal cover $\BH^2=\ti\gS_g$ yields an embedding of $\pi_1(\gS_g)$ in $\PSL_2(\BR)\cong\text{Isom}(\BH^2)^\circ$ and the image is a uniform lattice.  
\end{enumerate}
\end{examps}

\section
{Lattices resemble their ambient group in many ways.} Here are few illustrations of this phenomenon: Let $G$ be a locally compact group and $\gC\le_L G$ a lattice.
\begin{enumerate} 
\item $G$ is amenable iff $\gC$ is amenable.
\item $G$ has property (T) iff $\gC$ has property (T).
\item Margulis' normal subgroup theorem \cite[Ch. VIII]{Mar1}: If $G$ is a simple Lie group of real rank $\ge 2$ (e.g. $SL_n(\BR)$ for $n\ge 3$) then $\gC$ is just infinite, i.e. has no infinite proper quotients, or in other words, every normal subgroup of $\gC$ is of finite index.
\item Borel density theorem: If $G$ is a non-compact simple real algebraic group then $\gC$ is Zariski dense in $G$ (a good reference for this result is \cite{Fur}).
\end{enumerate}

Items $(1)$ and $(2)$ can be deduced directly from the definition of co-finiteness and we recommend them as exercises, but they can be found in many places (for an excellent reference for property $(T)$, see \cite{BHV}) . Note that both amenability and property $(T)$ can be expressed as fixed point properties. 

\section{Some basic properties of lattices}

In this section some basic results about lattices are proved. Students who wish to accomplish a more comprehensive background are highly encouraged to read \cite[Ch 1]{Rag} as well as \cite{Sel}. Let $G$ be a locally compact second countable group.

\begin{lem}[Compactness criterion]\label{thm:compt-crit}
Suppose $\gC\le_L G$, let $\pi:G\to G/\gC$ be the quotient map and let $g_n\in G$ be a sequence. Then $\pi(g_n)\to \infty$ (i.e. eventually leaves every compact set) iff there is a sequence $\gc_n\in\gC\setminus \{ 1\}$ such that $g_n\gc_ng_n^{-1}\to 1$. \end{lem}

Before proving this lemma, let us try to give an intuitive explanation. Thinking of $G/\gC$ as a generalisation of a manifold, the elements $\gc^g$ (with $\gc\ne 1$) correspond to the generalisation of (homotopically) nontrivial loops. Thus the lemma ``says" that a finite volume manifold is compact iff it admits arbitrarily short nontrivial loops.

\begin{proof}
Suppose that $\pi(g_n)$ does not go to infinity. Since $G$ is locally compact and, by definition of the quotient topology, $\pi$ is an open map, there is a bounded sequence $\{ h_n\}\subset G$ with $\pi(h_n)=\pi(g_n),~\forall n$. 
A subsequance $h_{n_k}$ converges to some $g_0$. Let $W$ be an identity neighborhood which intersects $\gC^{g_0}$ ($=g_0\gC g_0^{-1}$) trivially, and let $V$ be a symmetric identity neighborhood satisfying $V^3\subset W$. For sufficiently large $k$ we have $g_0h_{n_k}^{-1}\in V$ which implies that $\gC^{h_{n_k}}=\gC^{g_{n_k}}$ intersects $V$ trivially. 

Conversely, suppose that $\pi(g_n)\to \infty$. Let $W$ be a an arbitrary identity neighborhood in $G$ and let $V$ be a relatively compact symmetric identity neighborhood satisfying $V^2\subset W$. Let $K$ be a compact subset of $G$ such that $\vol(\pi(K))>\vol(G/\gC)-\mu (V)$. Since $\pi(g_n)\to\infty$, there is $n_0$ such that $n\ge n_0$ implies that $\pi(Vg_n)\cap\pi(K)=\emptyset$. The volumes inequality above then implies that $\vol(\pi(Vg_n))<\mu(V)$ and we conclude that $Vg_n$ does not inject to the quotient, i.e. that $Vg_n\cap Vg_n\gc\ne\emptyset$ for some $\gc\in\gC\setminus\{1\}$, hence $\gc^{g_n}\in V^2\subset W$. Since $G$ is second countable, one deduces that there are $\gc_n\in \gC$ such that $\gc_n^{g_n}\to 1$.
\end{proof}

\begin{exam}
Consider $\SL_2(\BZ)\le_L\SL_2(\BR)$, let \[
g_n=\left(\begin{array}{cc}
n & 0\\
0 & n^{-1}\end{array}\right)
\mbox{ and }\gc_n=\left(\begin{array}{cc}
1 & 0\\
1 & 1\end{array}\right)\]
then $g_n\gc_ng_n^{-1}\to 1$, and hence $\pi(g_n)\to\infty$.
\end{exam}

For general $G$, let us say that $h\in G$ is unipotent if the closure of its conjugancy class contains the identity.
Let us say that a sequence $\{h_n\}\subset G$ is asymptotically (or approximated) unipotent, if there are $g_n\in G$ such that $h_n^{g_n}\to 1$.

\begin{cor}
A lattice $\gC\le_L G$ admits non-trivial approximated unipotents in $G$ iff it is non-uniform.
\end{cor}

It is worth mentioning that a celebrated theorem of Kazhdan and Margulis \cite{Ka-Ma} states that if $G$ is a real algebraic semisimple group, then every non-uniform lattice in $G$ admits non-trivial unipotents.

\begin{exercise}
Let $G$ be a totally disconnected locally compact group and $\gC\le_L G$. Show that if $\gC$ is non-uniform then $\gC$ admits torsion, i.e. non-trivial elements of finite order. Moreover, show that in that case $\gC$ admits element of arbitrarily large finite order.
\end{exercise}

\noindent {\bf Hint:} Make use of V. Dantzig's theorem, namely that every totally disconnected locally compact group admits an open compact subgroup.

\medskip

We shall now explain some basic results established in the beautiful paper A. Selberg \cite{Sel}.

\begin{lem}[Recurrence]
Let $\gC\le_L G$, let $g\in G$ and let $\gO\subset G$ be an open set. Then $\gO^{-1} g^n\gO\cap\gC\ne\emptyset$ infinitely often.
\end{lem}

This is immediate from Poincare recurrence theorem, but let us sketch an argument:

\begin{proof}
Since $\pi(\gO)$ has positive measure while $\vol(G/\gC)$ is finite, we can find $k,m\in \BN$ with arbitrarily large gap, so that $g^k\cdot \pi(\gO)$ and $g^m\cdot \pi(\gO)$ are not disjoint. This means that $\pi(g^{m-k}\gO)\cap\pi(\gO)\ne\emptyset$ which is equivalent to $\gO^{-1} g^n\gO\cap\gC\ne\emptyset$ with
 $n=m-k$.
\end{proof}

\begin{exercise}[A week version of Borel's density theorem]\label{ex:borel-d}
Let $\gC\le_L\SL_n(\BR)$. Deduce from the last lemma that 
\begin{itemize}
\item $\gC$ admits regular elements, and
\item $\text{Span}(\gC)=M_n(\BR)$. (The less motivated student, may find the details in \cite{Sel}.)
\end{itemize} 
\end{exercise}

\begin{prop}\label{prop:cent}
Let $\gC\le_{UL} G$ (a uniform lattice) and $\gc\in\gC$. Let $C_G(\gc)$ be the centralizer of $\gc$ in $G$. Then $\gC\cap C_G(\gc)$ is a uniform lattice in $C_G(\gc)$.
\end{prop} 

\begin{proof}
There are two ways to prove this, one by constructing a compact fundamental domain for $\gC\cap C_G(\gc)$ in $C_G(\gc)$ and one by showing that the projection of $C_G(\gc)$ to $G/\gC$ is closed. Let us describe the first approach. 

Let $\gO$ be a relatively compact fundamental domain for $\gC$ in $G$. Let $\gd_1,\ldots,\gd_m\in \gC$ be chosen such that $\gd_i\gc\gd_i^{-1},~i=1,\ldots,m$ exhaust the finite set 
$\gO^{-1}\gc\gO\cap\gc^\gC$. We claim that 
$$
 \gO':=\cup_{i=1}^m\gO\gd_i\cap C_G(\gc)
$$ 
is a fundamental domain for $\gC\cap C_G(\gc)$ in $C_G(\gc)$.

Indeed, given $h\in C_G(\gc)$ we can express $h$ as $\go \gd$ with $\go\in\gO$ and $\gd\in\gC$, so $\gd=\go^{-1}h$ and we may find $1\le i\le m$ so that $\gc^\gd=w^{-1}\gc w=\gc^{\gd_i}$ (i.e. $\gd_i^{-1}\gd\in C_G(\gc)$). Thus 
$$
 h=(\go\gd_i)(\gd_i^{-1}\gd)\in\gO'\cdot C_\gC(\gc).
$$

\end{proof}

\begin{exercise}\label{ex:FCT}
Show (with the aid of Exercise \ref{ex:borel-d}) that if $\gC\le_{UL}\SL_n(\BR)$ then $\gC$ admits a diagonalizable subgroup isomorphic to $\BZ^{n-1}$. 
\end{exercise}

Here is a geometric interpretation of the last exercise: Let $X=G/K$ be the symmetric space of $\SL_n(\BR)$ (see Lecture $3$), and let $M=\gC\backslash X$ be a compact $X$-manifold (or orbifold). Then $M$ admits a flat totally geodesic imbedded $(n-1)$-torus. In fact any simple closed geodesic is contained in such a torus. This fact generalizes without difficulties to arbitrary Riemannian symmetric space $X$ where $n-1$ is replaced by rank$(X)=$ rank$_\BR(G)$. However, the analogous statement for general CAT(0) (or even Hadamard) spaces is the wide open, well known, Flat Closing Problem.

\begin{exercise}
Find an element $\gc$ in $\SL_3(\BZ)$ whose centraliser in $\SL_3(\BZ)$ is cyclic. Deduce that the analog of Proposition \ref{prop:cent} cannot hold for general non-uniform lattices. 
\end{exercise}

In spite of that, the analog of Exercise \ref{ex:FCT} does hold for non-uniform lattices as well, by a theorem of Prasad and Raghunathan \cite{PrRa}.

\section{A theorem of Mostow about lattices in solvable groups} 
The discussion in this section is taken from \cite{BCGM1}.

If $G$ is abelian and $H\le G$ a co-finite subgroup then $G/H$ is a group with finite Haar measure, hence compact. A similar result holds for general nilpotent groups:

\begin{prop}\label{prop:nilpotent}
Let $G$ be a nilpotent locally compact group and $H \leq G$ a closed subgroup. Then $H$ has finite covolume if and only if $H$ is cocompact.
\end{prop}

First prove:

\begin{exercise}\label{lem:nilpotent:unimod}
Every nilpotent locally compact group is unimodular.
\end{exercise}

We will also make use of the following:

\begin{exercise}
Let $G$ be a locally compact group and $H\le F\le G$ closed subgroups. If $H$ is co-finite then so is $F$. Furthermore if $F$ normalizes $H$ than $F/H$ is compact.\footnote{A more general, but slightly harder, statement is: For any locally compact groups $H\le F\le G$, $H$ is co-finite in $G$ iff it is co-finite in $F$ and $F$ is co-finite in $G$, see \cite[Lemma 1.6]{Rag}.}
\end{exercise}

%

\begin{proof}[Proof of Proposition~\ref{prop:nilpotent}]
Let us prove the ``only if" direction. Let $H \leq  G$ be a co-finite subgroup,  let $Z = Z(G)$ be the center of $G$ and let $F = \overline{Z\cdot H}$. Arguing by induction on the nilpotency degree, we infer that $F/Z$ is uniform in $G/Z$ and, hence, that $F$ is uniform in $G$. It is thus sufficient to prove that $H$ is uniform in $F$. The latter fact is clear since $H$ is normal in $F$ by definition.
\end{proof}

Let us note that the ``if" part of \ref{prop:nilpotent} holds in the much greater generality of amenable groups. Indeed, if $G$ is amenable and $H\le G$ is a closed uniform subgroup, the compact $G$-space $G/H$ (as any compact $G$-space) admits a $G$ invariant probability measure. The ``only if" direction however, is more involved for non-nilpotent groups. For solvable Lie groups Mostow proved the following classical result:

\begin{thm}[\cite{Mos1}]\label{mostow}
Let $G$ be a connected solvable Lie group. Then every co-finite subgroup of $G$ is uniform.
\end{thm} 
 
Let us give an elementary proof to Mostow's theorem.

\begin{exercise}\label{ex:slov}
\begin{itemize}
\item[$(a)$] Show that if $G$ is a connected Lie group and $\gC\le_L G$ is a finitely generated (or more generally compactly generated) abelian co-finite subgroup, then $\gC$ is uniform. (One can dig out an argument for this fact from the proof of Theorem 3.1 in \cite{Rag}, but I would recommend trying to establish a direct argument. In fact, this is true also for general locally compact $G$.) 
\item[$(b)$] Show that a connected solvable Lie group is Noetherian in the sense that every closed subgroup is compactly generated. (Hint: deduce the general case from the abelian case using induction on the solvability degree.)
\end{itemize}
\end{exercise}

\begin{proof}[Proof of Theorem \ref{mostow}]
We shall prove the result for every compactly generated solvable Lie group (note that a connected Lie group is compactly generated).
Let $G$ be a compactly generated solvable Lie group and $H\le G$ a co-finite subgroup. Up to replacing $G$ and $H$ by finite index subgroups we may assume that the commutator $G'$ is nilpotent. (Indeed, by the Ado--Iwasawa theorem $G$ admits an almost faithful complex linear representation $\rho:G\to\GL_d(\BC)$, and after replacing $G$ by a finite index subgroup, the Zariski closure of the image is connected, in which case, by Lie's theorem the commutator is nilpotent.)   
We may argue by induction on the nilpotency degree of $G'$, where the base case when $G'$ is trivial follows from Proposition \ref{prop:nilpotent}.
Let $Z$ be the center of $G'$. By induction $(G/Z)/(\overline{HZ}/Z)$ is compact, hence, we are left to show that $H$ is uniform in $E:=\overline{HZ}$, which is again compactly generated by Exercise \ref{ex:slov} $(b)$. Clearly $F=H\cap E'$ is normal in $E$. Dividing by $F$ we are left to prove that $H/F$ is cocompact in $E/F$. Since $H/F$ is abelian, the result follows from Exercise \ref{ex:slov} $(a)$.
\end{proof}

The analog of Mostow's theorem holds whenever $G$ is a linear group over a local field, and when $G$ is compactly generated with nilpotent commutator (see \cite{BCGM1}). However, in contrary to a conjecture of Benoist and Quint there are solvable groups which admits non-uniform lattices:

\begin{exam}[See \cite{BCGM1}]\label{exam:nonuniform-sol}
Let $p_n,~n\in\BN$ be primes such that $\prod_{n=1}^\infty\frac{p_n}{p_n-1}<\infty$. Let $G$ be the compact by discrete metabelian group
$$
 G=(\prod_{n=1}^\infty \BF_{p_n}^*)\ltimes (\bigoplus_{n=1}^\infty\BF_{p_n})
$$
and let $\gC$ be the set of sequences $(a_n,a_n-1)$ where $a_n\in\BF_{p_n}$ and $a_n=1$ for all but finitely manny $n$'s. It can be shown that $\gC$ is a non-uniform lattice in $G$.
\end{exam} 

\begin{exercise}[Completing details in the Example \ref{exam:nonuniform-sol}] 

$(1)$ Show that in the affine group $\BF \ltimes \BF^*$ over a field $\BF$, the set $\{(a,a-1):a\in \BF\}$ forms a subgroup. Deduce that $\gC$ is a subgroup of $G$.

$(2)$ Show that $\gC$ is discrete in $G$.

$(3)$ For $m\in\BN$ let 
$$
 G_m:=(\prod_{n=1}^\infty \BF_{p_n}^*)\ltimes (\bigoplus_{n=1}^m\BF_{p_n})~~~\text{and}~~~ \gC_m=\gC\cap G_m.
$$
Show that $[G_{m}:G_{m-1}]=p_m$ and $[\gC_{m}:\gC_{m-1}]=p_m-1$.

$(4)$ Deduce that if we normalise that Haar measure on $G$ so that its compact subgroup $\prod_{n=1}^\infty \BF_{p_n}^*$ has measure $1$ then 
$$
 \vol (G_m/\gC_m)=\frac{\vol(G_m)}{|\gC_m|}=\prod_{n=1}^m\frac{p_n}{p_n-1}.
$$

$(5)$ Making use of the data that $G$ is a direct limit of the open compact subgroups $G_n$, deduce that:
\begin{itemize}
\item $\gC$ is nonuniform in $G$ --- indeed, neither of the $G_N$ contains a fundamental domain for $\gC$ in $G$, since the sequence of co-volumes in $(4)$  does not stabilises,
\item $\vol (G/\gC)=\prod_{n=1}^\infty\frac{p_n}{p_n-1}<\infty$ and hence $\gC$ is a lattice.
\end{itemize} 
\end{exercise}

\section{Existence of lattices}
The very existence of (uniform and non-uniform) lattices is an interesting question. As we have seen in Exercise \ref{ex:L=>um} groups which are not unimodular cannot admit lattices. There are also examples of nilpotent Lie groups which admit no lattices (see \cite[Ch. 2]{Rag}). The discussion above shows that certain solvable groups admit only uniform lattices.
In \cite{BCGM2} it is shown that there are locally compact simple (amenable as well as non-amenable) groups which admit no lattices at all.
In the remaining lectures we will mostly restrict ourselves to the case where $G$ is a semisimple Lie group. By a classical theorem of Borel \cite{Bor} every connected semisimple Lie group admits plenty of uniform and non-uniform lattices. 

\section{Arithmeticity}
One of the highlights of the theory of lattices is the connection with arithmetic groups. This is illustrated in the following celebrated theorems\footnote{I recommend the book by D. Witti--Morris \cite{Wi} for an introduction to the theory of arithmetic groups}:

\begin{thm}[Borel--Harish-Chandra, see \cite{PR}]\label{thm:B-HC}
Let $\BG$ be an algebraic group defined over $\BQ$ which has no $\BQ$-characters. Then $\BG(\BZ)\le_L\BG(\BR)$. Furthermore,
$\BG(\BZ)\le_{UL}\BG(\BR)$ iff $\BG$ has no $\BQ$-co-characters.
\end{thm}

\begin{defn}
Let $G$ be a Lie group. We shall say that a subgroup $\gC\le G$ is arithmetic if there is a $\BQ$ algebraic group $\BH$ and a surjective homomorphism with compact kernel $f:\BH(\BR)\twoheadrightarrow G$ such that $f(\BH(\BZ))$ contains $\gC$ as a subgroup of finite index.
\end{defn}

If $G$ has no non-trivial real characters (homomorphisms to $\BR^*$), it follows that so does $\BH(\BR)$ and hence by Theorem \ref{thm:B-HC} that $\BH(\BZ)$ is a lattice in $\BH(\BR)$. Since $f$ has compact kernel, the image $f(\BH(\BZ))$ is still discrete, and hence a lattice in $G=f(\BH(\BR))$. 

\begin{exam}
Let $f(x,y,z)=x^2+y^2-\sqrt{2}z^2$, and consider the $\BQ[\sqrt{2}]$-group $\BG=\SO(f)$ and the subgroup $\gC=\BG(\BZ[\sqrt{2}])$. Let $\BH=\mathcal{R}_{\BQ[\sqrt{2}]/\BQ}\BG$ be the algebraic group obtained by restriction of scalars (see \cite[ 2.1.2]{PR}) and let $H=\BH(\BR)$. Then $H\cong \SO(2,1)\times SO(3)$ and $\gC$ is isomorphic to $\BH(\BZ)$. Since $\SO(3)$ is compact, $\gC$ projects faithfully to an arithmetic lattice in $\SO(2,1)$ and it can be shown that it has no unipotent. 
Now $\SO(2,1)$ acts properly by isometries on the hyperbolic plan $\BH^2$ and hence $\gC\backslash \BH^2$ is a compact hyperbolic orbiflod. This yields an interesting information about the algebraic structure of $\gC$. For instance, it is known that every such orbifold is finitely covered by a Riemann surface, and hence $\gC$ admits a finite index subgroup which is a surface group. Moreover, every surface group admits a presentation with $2g$ generators and one relator, where $g$ is the genus and can be computed from the area of the surface (the covolume of the corresponding lattice, with respect to an appropriate normalisation), using the Gauss--Bonnet theorem.
\end{exam}

Amazingly, in some cases the converse of Borel--Harish-Chandra theorem is also true. Recall that the rank of a Lie group is the minimal dimension of a centralizer of an element. The rank of $SL_n(\BR)$ is $n-1$ and the rank of $\SO(n,1)$ is one.

\begin{thm}[Margulis arithmeticity theorem \cite{Mar1}]\label{thm:arith}
If $G$ is a simple Lie group of rank $\ge 2$ then every lattice is arithmetic.
\end{thm}

In the fourth lecture we will discuss some rigidity theorems and describe how Margulis' superrigidity theorem implies the arithmeticity theorem. 
Superrigidity and arithmeticity hold also for lattices in the rank one groups $Sp(n,1)$ and $F_4^{-20}$ as proved in \cite{cor} and \cite{GS}.
On the other hand it is clear that $\SL_2(\BR)$ admits non-arithmetic lattices. Indeed, Teichmuller theory produces continuously many pairwise non-isometric hyperbolic structure on a surface of genus $g\ge 2$, yielding continuously many non-conjugate lattices, while it can be shown that only countably of them can be arithmetic. A beautiful construction of Gromov and Piatetsky-Shapiro \cite{Gr-PS} produces finite volume non-arithmetic real hyperbolic manifolds in every dimension $n\ge 3$, by gluing two pieces of arithmetic ones, proving that $\SO(n,1),~n\ge 2$ admit non-arithmetic lattices.


\lecture{On the Jordan--Zassenhaus--Kazhdan--Margulis theorem}

\section{Zassenhaus neighborhood}
Given two subsets of a group $A,B\subset G$ we denote by 
$$
 \{[A,B]\}:=\{[a,b]:a\in A,b\in B\}
$$ 
the set of commutators $[a,b]=aba^{-1}b^{-1}$. We define recursively $A^{(n)}:=\{[A,A^{(n-1)}]\}$ where $A^{(0)}:=A$.

By the Ado--Iwasawa theorem every Lie group is locally isomorphic to a linear Lie group. By explicit computation using sub-multiplicativity of matrix norms, one proves:

\begin{lem}\label{lem:Z1}
Every Lie group $G$ admits an open identity neighborhood $U$ such that $U^{(n)}\to 1$ in the sense that it is eventually included in every identity neighborhood.
\end{lem}

As pointed out, it is enough to explain this for $G=\GL_d(\BR)$. Write 
$$
 a=1+X,~b=1+Y
$$ 
with $X,Y\in M_d(\BR)$ and suppose $\|X\|\le\gep$ and $\|Y\|\le\gd\le\gep$. By continuity of the inverse map, for $\gep$ sufficiently small we have $\| a^{-1}\|,\|b^{-1}\|\le 2$. Thus by sub-multiplicity of the norm:
$$
 \| aba^{-1}b^{-1}-1\|=\| (ab-ba)a^{-1}b^{-1}\|=\| (XY-YX)a^{-1}b^{-1}\|\le 2\| X\|\cdot\| Y\|\cdot\| a^{-1}\|\cdot\| b^{-1}\|\le 8\gep\gd,
$$
hence for $\gep<\frac{1}{8}$ and $U=U_\gep:=\{ a\in\GL_d(\BR):\| a-1\|<\gep\}$ we see that $\gO^{(n)}$ tends to $1$ at an exponential speed.

\begin{exercise}
Let $\gD$ be a group generated by a set $S\subset\gD$. If $S^{(N)}=\{1\}$ for some $n$ then $\gD$ is nilpotent of class $\le N$.
\end{exercise}

\begin{cor}
If $\gD\le G$ is a discrete subgroup then $\langle\gD\cap U\rangle$ is nilpotent.
\end{cor} 
 
\begin{proof}
Since $\gD$ is discrete, there is an identity neighborhood $V$ which intersects $\gC$ trivially.
By Lemma \ref{lem:Z1} $S:=\gD\cap U$ satisfies $S^{(n)}\to 1$, hence for some $N$, $S^{(N)}\subset V\cap\gD=\{1\}$ and the result follows from the previous exercise.
\end{proof} 
 
Furthermore, taking $\gO=U_\gep$ with sufficiently small $\gep$ we can even guarantee that every discrete group with generators in $\gO$ is contained in a {\it connected} nilpotent group:

\begin{thm}[Zassenhaus (1938), Kazhdan--Margulis (1968)]
Let $G$ be a Lie group. There is an open identity neighborhood $\gO\subset G$ such that every discrete subgroup $\gD\le G$ which is generated by $\gD\cap \gO$ is contained in a connected nilpotent Lie subgroup of $N\le G$. Moreover $\gD\le_{UL} N$.
\end{thm}

The idea is that near the identity the logarithm is well defined and two elements commute iff their logarithms commute. 
For a complete proof see \cite[Theorem 8.16]{Rag} or \cite[Section 4.1]{Th}.

A set $\gO$ as in the theorem above is called a Zassenhaus neighborhood.

\section{Jordan's theorem}
Since connected compact nilpotent groups are abelian, we deduce the following classical result:

\begin{thm}[Jordan 1878]\label{thm:jordan}
For a compact Lie group $K$ there is a constant $m\in \BN$ such that every finite subgroup $\gD\le K$ admits an abelian subgroup of index $\le m$. 
\end{thm}

\begin{proof}
Let $\gO$ be a Zassenhaus neighborhood in $K$, let $U$ be a symmetric identity neighborhood satisfying $U^2\subset\gO$, and set $m:=\frac{\mu(K)}{\mu(U)}$. Given a finite subgroup $F\le K$ set $A=\langle F\cap\gO\rangle$. By the remark preceding the theorem $A$ is abelian. Now if $f_1,\ldots,f_{m+1}$ are $m+1$ elements in $F$ then for some $1\le i\ne j\le m+1$ we have $f_iU\cap f_jU\ne\emptyset$ implying that $f_i^{-1}f_j\in F\cap\gO\subset A$. Thus $[F:A]\le m$.
\end{proof}

Note that since any connected Lie group $G$ admits a unique maximal compact subgroup $K$ up to conjugation, one can state Jordan's theorem for non-compact connected Lie groups as well. (Originally, it was stated for $G=\GL_n(\BC)$.) 

\section{Approximations by finite transitive spaces}
Let us make a short detour before continuing the discussion about discrete groups. Suppose that $K$ is a metric group. An $\gep$-quasi morphism $f:F\to K$ from an abstract group $F$ is a map satisfying $d(f(ab),f(a)f(b))\le \gep,~\forall a,b\in F$. We shall say that $K$ is quasi finite if for every $\gep$ there a finite group $F$ and an $\gep$-quasi morphism into $K$ with an $\gep$-dense image (i.e. $\forall k\in K,\exists a\in F$ with $d(f(a),k)\le\gep$). Relying on Jordan's theorem, Turing showed \cite{Turing}:

\begin{thm}[Turing 1938]\label{thm:turing}
A compact connected Lie group is quasi finite iff it is abelian (i.e. a torus).
\end{thm}

Recall that a metric space is said to be transitive if its isometry group acts transitively.
With the aid of Turing's theorem one can classify the metric spaces which can be approximated by finite transitive ones:

\begin{thm}[\cite{Ge3}]\label{thm:metric-turing}
A metric space is a limit of finite transitive spaces (in the Gromov--Hausdorff topology) iff it admits a transitive compact group of isometries whose identity connected component is abelian. 
\end{thm}

The lines of the proof are as follows. Given a metric space $X$, one shows that there is a $\gd_0>0$ and a function $\gep:(0,\gd_0)\to \BR^{>0}$ whose limit at $0$ is $0$, such that for any finite metric space $\mathcal{F}$ with $d_{GH}(X,\mathcal{F})<\gd\le\gd_0$ there is a natural $\gep(\gd)$ quasi morphism from the finite group $\text{Isom}(\mathcal{F})$ to  $\text{Isom}(X)$. The result is then proved relying on structure theorems for compact groups and on \ref{thm:turing} (see \cite{Ge3} for details).

It follows from Theorem \ref{thm:turing} and the Peter--Weyl theorem that if $X$ is approximable by finite transitive spaces then its connected components are inverse limits of tori, hence the only manifolds that can be approximated are tori. In particular we obtain the following result which answers a question of I. Benjamini and can be interpreted as the non-existence of a perfect soccer ball: 

\begin{cor}
$S^2$ cannot be approximated by finite homogeneous spaces.
\end{cor}

Theorem \ref{thm:metric-turing} has also some graph theoretic applications. For instance one can deduce that any sequence of distance-transitive graphs with normalized diameter and bounded geometry converges, in the Gromov--Hausdorff sense, to a circle (see \cite[Corollary 1.6]{Ge3}). 

\section{Margulis' lemma}
Coming back from this short detour, let us present another classical result:

\begin{thm}(The Margulis lemma, \cite[Section 4.1]{Th})\label{MarLem}
Let $G$ be a Lie group acting properly by isometries on a Riemannian manifold $X$. Given $x\in X$ there are $\gep=\gep(x)>0$ and $m=m(x)\in\BN$ such that if $\gD\le G$ is a discrete subgroup which is generated by the set 
$$
 \gS_{\gD,x,\gep}:=\{\gc\in\gD:d(\gc\cdot x,x)\le\gep\}
$$
then $\gD$ admits a subgroup of index $\le m$ which is contained in a (closed) connected nilpotent Lie group. Furthermore, if $G$ acts transitively on $X$ then $\gep$ and $m$ are independent of $x$. 
\end{thm}

\begin{proof}
The properness of the action implies that the set 
$$
 C=\{ g\in G: d(g\cdot x,x)\le 1\}
$$ 
is compact. Let $V\subset G$ be a relatively compact open symmetric set such that $V^2$ is a Zassenhaus neighborhood. Setting 
$$
 m=[\frac{\vol(C\cdot V)}{\vol(V)}]~\text{and}~\gep=1/m
$$ 
one can prove the theorem arguing as in the proof of \ref{thm:jordan}. An extra complication arises from the fact that $\gD$ and $F=\langle\gD\cap V^2\rangle$ are infinite, but this can be taken care of by observing that whether a connected graph has more than $m$ vertices or not, can be seen by looking at a ball of radius $m$ in the graph. Thus, assuming in contrary that the Schreier graph of $\gD/F$ has more than $m$ vertices, we could find $m+1$ elements $\gc_1,\ldots,\gc_{m+1}$ in the $m$-ball $(\gS_{\gD,x,\gep})^m$ which belong to mutually different cosets of $F$. However, by the choice of $\gep$ we have that $(\gS_{\gD,x,\gep})^m\subset C$, hence for some $1\le i\ne j\le m+1$ we have $\gc_iV\cap\gc_jV\ne\emptyset$, i.e. $\gc_i^{-1}\gc_j\in V^2\cap\gD\subset F$, a contradiction.
\end{proof}

A differential--geometric proof of the Margulis lemma, which provides more information, can be found in \cite{BGS}. 

\section{Crystallographic manifolds}
In the special case of $X=\BR^n$ since homotheties commute with isometries it follows that $\gep=\infty$ --- i.e. that any finitely generated\footnote{Since any discrete subgroup of $\text{Isom}(\BR^n)$ is f.g. this assumption is in fact redundant.} discrete group of isometries of $\BR^n$ is virtually (i.e. admits a finite index subgroup which is) contained in a connected nilpotent group. Indeed, given any $\gep>0$ and a finite set $\gS$ generating a discrete subgroup of $\text{Isom}(\BR^n)$, one can rescale the metric (or alternatively, apply a homothety) so that the displacement of $\gS$ at an arbitrary point $x\in\BR^n$ becomes less than $\gep$.

Moreover, it is easy to verify that the connected nilpotent subgroups of the group $G=\text{Isom}(\BR^n)\cong O_n(\BR)\ltimes \BR^n$ are abelian. Given an isometry $\gc$ of $\BR^n$ one can decompose $\BR^n$, considered as an affine space, to $\min(\gc)\oplus\min(\gc)^\perp$ where $\min(\gc)$ is the affine subspace on which $\gc$ acts by a translation and $\min(\gc)^\perp$ is an arbitrarily located orthogonal complement.
Clearly for $\gc$ non-elliptic (i.e. which does not fix a point) $\min(\gc)$ has positive dimension, and it is not hard to show that if $\Lambda$ is a set of commuting non-elliptic isometries then $\cap_{\gc\in\Lambda}\min(\gc)$ has positive dimension and is $\langle\Lambda\rangle$-invariant.  

\begin{exercise} Complete the details above as follows: 

\begin{enumerate}
\item Show that every connected nilpotent subgroup of $O_n(\BR)\ltimes \BR^n$ is abelian. (Hint: use the fact that a compact connected nilpotent group is abelian.)

\item Show that an isometry of $\BR^n$ whose linear part has no nonzero invariant vector must have a fixed point. Deduce the existence of the above decomposition $\min(\gc)\oplus\min(\gc)^\perp$ by first decomposing $\BR^n$ as a linear space according to the liner part of $\gc$.

\item Show that the min-sets of arbitrarily many commuting isometries intersect nontrivially.

\end{enumerate}

\end{exercise}

Thus, we deduce:

\begin{thm}[Bieberbach (1911) --- Hilbert's 18'th problem]\label{thm:bie}
Let $\gC$ be a torsion free group acting properly discontinuously by isometries on $\BR^n$. Then $\gC$ admits a finite index subgroup isomorphic to $\BZ^k$ ($k\le n$) which acts by translations on some $k$ dimensional invariant subspace, and $k=n$ iff $\gC$ is uniform. In particular, every crystallographic manifold is finitely covered by a torus. 
\end{thm}

A well known result, commonly attributed to Selberg states that every finitely generated linear group is virtually torsion free (c.f. \cite[Corollary 6.13]{Rag}). Thus Theorem \ref{thm:bie} holds without the assumption that the discrete group $\gC\le\text{Isom}(\BR^n)$ is torsion free.


\lecture{On the geometry of locally symmetric spaces and some finiteness theorems}
\section{Hyperbolic spaces}
Consider the hyperbolic space $\BH^n$ and its group of isometries $G=\text{Isom}(\BH^n)$. Recall that $G^\circ\cong\text{PSO}(n,1)$ is a rank one simple Lie group. For $g\in G$ denote by 
$$
 d_g(x):=d(g\cdot x,x)
$$ 
the displacement function of $g$ at $x\in \BH^n$.
Let 
$$
 |g|=\inf d_g~\text{and}~\min(g)=\{ x:\in \BH^n: d_g(x)=|g|\}.
$$ 
Note that $d_g$ is a convex function which is smooth outside $\min(g)$.

The isometries of $\BH^n$ split to 3 types (c.f. \cite[Section 2.5]{Th}):
\begin{itemize}
\item {\it elliptic} --- those that admit fixed points in $\BH^n$.
\item {\it hyperbolic} --- isometries for which $d_g$ attains a positive minimum. In that case $\min (g)$ is a $g$-invariant geodesic, called the axis of $g$.
\item {\it parabolic} --- isometries for which $\inf d_g=0$ but have no fixed points in $\BH^n$.
\end{itemize}

The first two types are called {\it semisimple}. 

One way to prove that every isometry is of one of these forms is to consider the visual compactification $\overline\BH^n=\BH^n\cup\partial\BH^n$, where $\partial\BH^n$ can be defined as the set of geodesic rays up to bounded distance (the student is refereed to P.E. Caprace's course --- given in parallel --- for a detailed description of this compactification). The action of $G$ on $\BH$ extends to a continuous action on $\overline\BH^n$ and $\overline\BH^n$ is homeomorphic to a closed ball in $\BR^n$. By Brouwer's fixed point theorem every $g\in G$ admits a fixed point in $\overline\BH^n$. 

\begin{exercise}
Suppose $n\ge 2$ and let $g\in G$.
\begin{itemize}
\item If $g$ has $3$ fixed points on $\partial\BH^n$ then $g$ fixes point wise the hyperbolic plane in $\BH^n$ determined by these $3$ points, and in particular $g$ is elliptic.
\item If $g$ is non-elliptic and has exactly $2$ fixed points at $\partial\BH^n$ then $g$ is hyperbolic and its axis is the geodesic connecting these fixed points.
\item If $g$ has exactly one fixed point at $\partial\BH^n$ then $g$ is parabolic.
\end{itemize}
\end{exercise}

By considering the upper half space model for $\BH^n$ it is easy to see that a parabolic isometry preserves the horospheres around its fixed point at infinity, and that each such horosphere, considered with its intrinsic metric, is isometric to $\BR^{n-1}$.

\begin{exercise}\label{Ex:commute}
Suppose that $g,h\in G$ commute, then
\begin{itemize}
\item if $g$ is hyperbolic, then $h$ is semisimple,
\item if $g$ and $h$ are both hyperbolic then they share a common axis,
\item if $g$ and $h$ are parabolics, they have the same fixed point at $\partial \BH^n$.
\end{itemize}
\end{exercise}

\begin{exercise}
A discrete subgroup $\gD\le G$ admits a common fixed point in $\BH^n$ if and only if it is finite. 
\end{exercise}

It follows that a discrete group $\gC\le G$ acts freely on $\BH^n$ if and only if it is torsion free.

\section{The thick--thin decomposition}
Let $\gC\le G$ be a torsion free discrete subgroup. We denote by $M=\gC\backslash \BH^n$ the associated complete hyperbolic manifold. Note that $\gC$ is a lattice iff $M$ has finite volume. We denote by $\text{InjRad}(x)$ the injectivity radius at $x$.
Let $\gep(\BH^n)$ be the Margulis' constant of $\BH^n$ (see Theorem \ref{MarLem} in the previous lecture) and set $\gep=\frac{1}{10}\gep(\BH^n)$ (it is helpful for some arguments to work with a constant which is strictly smaller than $\gep(\BH^n)$). Let 
$$
 M_{<\gep}=\{ x\in M: \text{InjRad}(x)<\gep/2\},~\text{and}~M_{\ge\gep}=\{ x\in M: \text{InjRad}(x)\ge\gep/2\}
$$ 
be the $\gep$-thin part and the $\gep$-thick part of $M$.

\begin{exercise}
Show that if $\ti x$ is a lift of $x\in M$ in $\BH^n$ then 
$$
 \text{InjRad}(x)=\frac{1}{2}\min\{ d_\gc(\ti x):\gc\in\gC\setminus\{1\}\}.
$$
\end{exercise}

\begin{thm}(The thick--thin decomposition \cite[Section 4.5]{Th})
Suppose that $\vol(M)<\infty$.
Then each connected component $M_{<\gep}^\circ$ of the thin part $M_{<\gep}$ is either
\begin{itemize}
\item a {\it tubular neighborhood of a short closed geodesic}, in which case $M_{<\gep}^\circ$ is homeomorphic to a ball bundle over a circle, or
\item a {\it cusp}, in which case $M_{<\gep}^\circ$ is homeomorphic to $N\times\BR^{>0}$ where $N$ is some $(n-1)$-crystallographic manifold.
\end{itemize}
The number of connected components of $M_{<\gep}$ is at most $C\cdot\vol(M)$ for some constant $C=C(\BH^n)$, and in case $n\ge 3$, the thick part $M_{\ge\gep}$ is connected. 
\end{thm}

\begin{proof}
Let $\ti M_{<\gep}$ be the pre-image in $\BH^n$ of the thin part of $M$. Observe that 
$$
 \ti M_{<\gep}=\cup_{\gc\in\gC\setminus\{1\}}\{d_\gc<\gep\}
$$ 
is the union of the $\gep$-sub-level sets of the functions $d_\gc$. Note that a sub-level set of the displacement function of a hyperbolic isometry is a convex neighborhood of the axis and a sub-level set of a parabolic isometry is a convex neighborhood of the fixed point at infinity. Let $M_{<\gep}^\circ$ be a connected component of $M_{<\gep}$ and $\ti M_{<\gep}^\circ$ a connected component of its pre-image in $\ti M_{<\gep}$. Then $\ti M_{<\gep}^\circ$ is a $\gC$-precisely invariant set, i.e. for $\gc\in \gC$, either $\gc\cdot \ti M_{<\gep}^\circ=\ti M_{<\gep}^\circ$ or $\gc\cdot \ti M_{<\gep}^\circ\cap\ti M_{<\gep}^\circ=\emptyset$. Let $\gC^\circ=\{\gc\in\gC:\gc\cdot \ti M_{<\gep}^\circ=\ti M_{<\gep}^\circ\}$. Then $M_{<\gep}^\circ=\gC^\circ\backslash \ti M_{<\gep}^\circ$.

Consider $\ga,\gb\in\gC\setminus\{1\}$ such that $\{d_\ga<\gep\}\cap\{d_\gb<\gep\}\ne\emptyset$. By the Margulis lemma, for some $1\le i,j\le m$ the group $\langle \ga^{i},\gb^{j}\rangle$ is nilpotent. Let $\eta$ be a non-trivial central element in this group. By Exercise \ref{Ex:commute} if $\eta$ is hyperbolic then so are $\ga$ and $\gb$ and they all have the same axis, and if $\eta$ is parabolic then so are $\ga,\gb$ with the same fixed point at infinity. It follows that $\ti M_{<\gep}^\circ$ is of the form $\cup_{\gc\in I}\{d_\gc<\gep\}$ where $I$ consists either of hyperbolic elements sharing the same axis or of parabolic elements fixing a common fixed point at infinity. In the first case, the discreteness of the torsion free group $\gC$ implies that the element $\gc_0$ in $I$ with minimal displacement generates $\gC^\circ$, and in particular, the later is a cyclic group. In the second case, $\gC^\circ$ preserves the horospheres around the fixed point at infinity, and it is not hard to see that $M_{<\gep}^\circ$ is homeomorphic to the quotient of a horoball by $\gC^\circ$. (In this case one can actually take $I=\gC^\circ\setminus\{1\}$.) The assumption $\vol (M)<\infty$ implies that the quotient of the boundary horosphere by the group $\gC^\circ$ must be compact, hence a crystallographic manifold.

In dimension $n\ge 3$ it follows that the boundary of each thin component is connected, since the boundary of a tube is an $(n-2)$-sphere bundle over a circle, hence connected, while the boundary of a cusp is always connected. This implies that the connectedness of the thick part $M_{\ge\gep}$.

Finally, in order to bound the number of components of $M_{<\gep}$ note that one can attach disjointly injected $\gep$-balls near the boundary of every component. (One of the reasons for chosing $\gep$ strictly smaller than the Margulis' constant was to make sure that the thin component are not too close to one another.)
\end{proof}

\begin{rem}
An analog result holds (with minor changes in the proof) for every rank one symmetric space (and even for every negatively curved Hadamard space).
\end{rem}

\section{Presentations of torsion free lattices}
The thick-thin decomposition is an important ingredient in the proof of the following:

\begin{thm}[\cite{BGLM,Ge1}]\label{thm:presentation}
There is a constant $c=c(G)$ such that every
torsion free lattice $\gC\le_L G$ admits a presentation $\gC=\langle \gS | R\rangle$ with $|\gS|,|R|\le c\cdot\vol (G/\gC)$. Furthermore, unless $G\cong\PSL_2(\BC)$ there is such a presentation in which the length of every relation is at most $3$. 
\end{thm}

Let me explain the idea of the proof. The case $n=2$ is well known --- $\gC$ is a surface  group on $2g$ generators, where $g$ is the genus of the surface $\gC\backslash\BH^2$ and $-\chi=2g-2$ is proportional to the volume by the Gauss--Bonnet theorem, and one relation of length $4g$. With the price of adding $<4g$ generators, we can "break" the relation to $<4g$ piece of length $3$.

Suppose $n> 3$. Gluing the thin components to the thick part one by one and using the Van--Kampen theorem, one sees that $\pi_1(M)\cong \pi_1(M_{\ge\gep})$. Indeed, when gluing a cusp the homotopy type is unchanged (this is also the case in dimension $3$) while when gluing a tubular neighborhood of a closed geodesic we see that
$$
 \pi_1(M_{\ge\gep}\cup M_{<\gep}^\circ)\cong \pi_1(M_{\ge\gep})*_\BZ\BZ,
$$
where these $\BZ$'s correspond to the fundamental groups of $M_{<\gep}^\circ$ and of its boundary, hence the map between them is an isomorphism, and they cancel each other.

Now $M_{\ge\gep}$, being an $\gep$-thick manifold (forget for a moment the boundary), can be covered by $\frac{\vol(M)}{\vol(B_{\gep/2})}$ balls of radius $\gep$ with bounded overlaps --- this can be done by taking the centers of these balls to form a maximal $\gep$-discrete set. The nerve of a cover is a simplicial complex whose cells correspond to collections of sets of the cover which have a common nonempty intersection. Thus the vertices of the nerve bijectively correspond to the sets of the cover, the edges correspond to pairs of sets which are not disjoint, etc. Since balls in $\BH^n$ are convex, and intersections of convex sets are still convex and hence contractible, it follows from \cite[Theorem 13.4]{BT}) that the nerve $\mathcal{R}$ of our cover is 
homotopic to $M_{\ge\gep}$.

Let me describe how to give an efficient presentation to $\pi_1(\mathcal{R})$.
We have already noted that $\mathcal{R}$ has at most $C\cdot\vol (M)$ vertices with $C=\frac{1}{\vol(B_{\gep/2})}$. Since neighbouring vertices corresponds to balls whose centres are at distance at most $2\gep$, and since the $\gep/2$-balls around the centres of our cover are disjoint it follows that the degree at any vertex is at most 
$d=[\frac{\vol(B_{2.5\gep})}{\vol(B_{\gep/2})}]$. 
Fix a spanning tree $T$ for 
$\mathcal{R}$, and take the generating set $\gS$ for $\pi_1(\mathcal{R})\cong\pi_1(M)$ 
which consists those closed loops which contain exactly one edge outside $T$.
We thus obtain a generating set of size less then the number of edges of the 
1-skeleton $\mathcal{R}^1$ which is at most $\frac{C\vol (M)d}{2}$.
In other words, we take for each edge of $\mathcal{R}^1\setminus T$ the element
of $\pi_1(\mathcal{R}^1)$ which corresponds to the unique cycle (with 
arbitrarily chosen orientation) which is obtained by adding this edge to $T$.
Additionally, let the set of relations $W$ consists of exactly 
those words which are induced from 2-simplexes of $\mathcal{R}^2$ (we take
one such relation for each 2-simplex). In this way we obtain
a set of relations of size less than $C\vol (M) d^2$ which is a bound for
the number of triangles in $\mathcal{R}^1$.
Finally, the length of each $w\in W$ is exactly the number of edges in the 
corresponding 2-simplex which lie outside $T$, i.e. at most $3$.

For $n=3$ (i.e. when $G\cong\PSL_2(\BC)$) we get that the fundamental group of the boundary $\partial M_{<\gep}^\circ$ of a tubular neighborhood of a closed geodesic is $\BZ^2$ and hence 
$$
 \pi_1(M_{\ge\gep}\cup M_{<\gep}^\circ)\cong \pi_1(M_{\ge\gep})*_{\BZ^2}\BZ,
$$
and the kernel of the map between $\BZ^2$ to $\BZ$ is the cause of relations of an uncontrolled length. 

\begin{rem}\label{rem:bounded-presentation}
\begin{enumerate}
\item
The rigorous proof of the Theorem \ref{thm:presentation} is in fact much more involved, due to the textured structure of the boundary of $M_{\ge\gep}$ (see \cite{BGLM,Ge1} for details).

\item
A similar theorem holds for every non-compact semi-simple Lie group $G$, with the exceptions of $\PSL_2(\BC),\PSL_3(\BR)$ and $\PSL_2(\BR)^2$, but the proof is more complicated (see \cite{Ge1}). While the analogous statement is evidently fouls for $\PSL_2(\BC)$, it conjecturally holds for the other two exceptions, however the current proofs do not apply in these cases.
\end{enumerate}
\end{rem}

Theorem \ref{thm:presentation} is a week version of the following general conjecture, suggesting that the homotopy complexity of locally symmetric spaces are bounded linearly by their volume:

\begin{conj}\label{conj:HV}
Let $X$ be a symmetric space of non-compact type (see Definition \ref{de:symm-space}) and suppose that $\dim X\ne 3$. Then there are constants $\ga$ and $d$, depending only on $X$, such that every irreducible complete Riemannian manifold $M$ locally  isometric to $X$ is homotopically equivalent to a simplicial complex $\mathcal{R}$ with at most $\ga\cdot\vol(M)$ vertices and all the vertices degrees are bounded by $d$. 
\end{conj}

For the special case of {\it non-compact arithmetic} locally symmetric spaces, Conjecture \ref{conj:HV} has been confirmed in \cite{Ge1}. 

\section{General symmetric spaces}

\begin{defn}\label{de:symm-space}
A symmetric space is a complete Riemannian manifold $X$ such that for every $p\in X$ there is an isometry $i_p$ which fixes $p$ and reflects the geodesics through $p$.  
\end{defn}

A symmetric space admits a canonical De Rham decomposition $X=\prod X_i$ to irreducible factors. We shall say that $X$ is of non-compact type if each of the $X_i$ is neither compact nor $\cong\BR$. In that case $\text{Isom}(X)^\circ$ is a center-free semisimple Lie group without compact factors. Conversely, if $G$ is a connected center-free semisimple Lie group without compact factors then $G$ admits a, unique up to conjugacy, maximal compact subgroup $K$ and $G/K$ admits a canonical metric with respect to which it is a symmetric space of non-compact type with isometry group whose identity component is $G$. Symmetric spaces of non-compact type are non-positively curved, i.e. they are CAT(0), or equivalently the distance function $d:X\times X\to\BR^{\ge 0}$ is convex.
A {\it flat subspace} of $X$ is a totally geodesic subspace isometric to a Euclidian space. A {\it flat} is a maximal flat subspace. $G$ acts transitively on the set of flats and in fact on the set of pairs $(x,F)$ consisting of a flat $F$ and a point $x\in F$, but $G$ does not act transitively on flat subspaces of a given positive but not maximal dimension. The rank of $X$ is the dimension of a flat, and is equal to the algebraic rank of $G$. $X$ is strictly negatively curved iff $\text{rank}(X)=1$. As in the rank one case, there are three types of isometries: elliptic ($\text{Fix}(g)\ne\emptyset$), parabolic ($\min(g)=\emptyset$) and hyperbolic ($\min(g)\ne\emptyset=\text{Fix}(g)$).\footnote{We keep the same notations for displacement functions, min-sets, etc. as introduced in the hyperbolic spaces case.} 
We refer the reader to \cite[Appendix 5]{BGS} for a very nice and short exposition of symmetric spaces, and to \cite{Eb} for a much more exhaustive treatment. 

\begin{exam}
As a model for the symmetric space of $G=\PSL_n(\BR)$, denoted $P^1(n,\BR)$, we can take the space of all unimodular positive definite $n\times n$ matrices on which $G$ acts by similarity: $g\cdot p:= gpg^t$. The tangent space at $I$ is the space of trace $0$ symmetric $n\times n$ matrices. The inner product at $T_I(P^1(n,\BR))$ is given by $\langle X,Y\rangle:=\text{trace}(XY)$, the geodesics through $I$ are of the form $\exp(tX)$ and the curvature at $X,Y\in T_I(P^1(n,\BR))$ is given by $K(X,Y)=-\| [X,Y]\|$.
\end{exam}

As in the special hyperbolic case, there is a one to one correspondence between discrete subgroups $\gC$ of $G$ and complete $X$-orbifolds $M=\gC\backslash X$, where torsion free groups correspond to manifolds, and lattices corresponds to orbifolds of finite volume. Moreover, up to re-normalizing the Haar measure on $G$ we may suppose that it corresponds to the Riemannian measure on $X$, in the sense that $\vol(G/\gC)=\vol(\gC\backslash X)$.

\section{Number of generators of lattices}
We wish to push further the philosophy that one can analyze properties of $\gC$ by studying the topology of $M=\gC\backslash X$.
When $\gC$ has torsion, $M$ has ramified points, i.e. it is not a manifold, and dealing with the geometry of orbifolds is much more delicate. Still, using some basic Morse theory, the nonpositivity of the curvature and the Margulis' lemma, one can prove:

\begin{thm}[\cite{Ge2}]\label{thm:d(Gamma)}
Let $G$ be a connected semisimple Lie group without compact factors. There is a constant $C=C(G)$ such that $d(\gC)\le C\cdot \vol(G/\gC)$ for every discrete group $\gC\le G$, where $d(\gC)$ denotes the minimal cardinality of a generating set.
\end{thm}

Let us outline the idea of the proof.
For $G=\PSL_2(\BR),~X=\BH^2$ the theorem can be deduced from the Gauss--Bonnet theorem, so let us assume that $n\ge 3$.

\begin{lem}
Let $X$ be an irreducible symmetric space of dimension $>2$. Let $g\in G=\text{Isom}(X)^\circ$ be a non-trivial element. Then $\dim(X)-\dim (\min(g))\ge 2$.
\end{lem}

In the rank one case this is obvious. Indeed if $g$ is parabolic $\min(g)=\emptyset$, if $g$ is hyperbolic $\min(g)$ is one dimensional and if $g$ is non-trivial elliptic, $\min(g)$ must have co-dimension at least two for otherwise it is orientation reversing. For general symmetric space one can produce, from the existence of an element $g\in G$ with $\text{codim}(\min(g))=1$ a $G$ action on the circle (see \cite[Lemma 2.1]{Ge2}), while it is well known that the only simple Lie group that acts non-trivially on the circle is $\PSL_2(\BR)$. 

It follows that 
$$
 \ti Y:=X\setminus\cup\{\min(\gc):\gc\in\gC\setminus\{1\}\}
$$ 
is a connected $\gC$-invariant subset of $X$. Let $Y=\gC\backslash \ti Y$ be the image of $\ti Y$ in $M$.

Let $f:\BR^{>0}\to\BR^{\ge 0}$ be a smooth function which tends to $\infty$ at $0$, is strictly decreasing on $(0,\gep]$ and is identically $0$ on $[\gep,\infty)$. Let $\gC_\circ=\{\gc\in\gC\setminus\{1\}:|\gc|\le\gep\}$. Define $\ti\psi:\ti Y\to\BR$ as follows,
$$
 \ti\psi(x)=\sum_{\gc\in\gC_\circ} f(d_\gc(x)-|\gc|).
$$
Note that $\ti\psi$ is well defined (and smooth) since for every $x\in \ti Y$ only finitely many of the summands are nonzero as $\gC$ is discrete.
Clearly $\ti \psi$ is $\gC$-invariant and hence induces a map $\psi:Y\to \BR^{\ge 0}$.

\begin{lem}[Main lemma]\label{lem:main}
The gradient of $\psi$ vanishes precisely where $\psi$ vanishes.
\end{lem}

The following exercises may help in understanding the proof of the lemma:
\begin{exercise}
Let $X$ be a CAT(0) space (e.g. a symmetric space), let $A\subset X$ be a closed convex subset and let $g$ be an isometry of $X$ preserving $A$. Show that:

$(1)$ The nearest point projection $\pi_A:X\to A$ is $1$-Lipschitz.

$(2)$ For every $x\in X$, we have $d_g(x)\ge d_g(\pi_A(x))$.

$(3)$ If $g$ is semisimple then $\min(g)\cap A\ne\emptyset$.

$(4)$ Suppose that $g_1$ and $g_2$ are commuting isometries of $X$ and $t_1,t_2$ are positive numbers such that the sub-level sets for the displacement functions $\{ d_{g_i}\le t_i\}$ are both non-empty, then
$\{ d_{g_1}\le t_1\}\cap\{ d_{g_2}\le t_2\}\ne\emptyset$.

%

\end{exercise}

\begin{proof}[Proof of Lemma \ref{lem:main}]
For the sake of simplicity let us assume that $X$ is the hyperbolic space $\BH^n$, but the same ideas extend to the general case. 
At any point $x\in \ti Y$ with $\ti\psi (x)\ne 0$ we will find a tangent vector $\hat{n}_x$ at which the directional derivative of $\ti\psi$ is nonzero. Let 
$$
 \gS_x=\{\gc\in\gC_\circ: f(d_\gc(x)-|\gc|)\ne 0\},~\text{let}~\gD_x=\langle \gS_x\rangle
$$
and let $N_x$ be a normal subgroup of finite index in $\gD_x$ which is contained in a connected nilpotent Lie subgroup of $G$. In view of Selberg's lemma (that every finitely generated linear group is virtually t.f.) we may also suppose that $N_x$ is torsion free. Let $Z_x$ denote the center of $N_x$. We distinguish between 3 cases. 

\medskip

{\bf Case 1:} Suppose first that $\gD_x$ is finite. Let $y\in X$ be a fixed point for $\gD_x$ and let $\hat{n}_x$ be the unit tangent at $x$ to the geodesic ray $c:[0,\infty)\to X$ emanating from $y$ through $x$. Thus $\hat{n}_x=\dot{c}(d(x,y))$. Since $x\in \ti Y$ it follows that $d_\gc(x)>0,~\forall \gc\in\gC\setminus\{1\}$, and since $d_\gc\circ c$ is a convex function (as $X$ is non-positively curved) we deduce that 
$$
 \frac{d}{dt}|_{t=d(x,y)}d_\gc(c(t))>0,
$$
for every $\gc\in\gS_x$ and hence
$$
 \nabla\ti\psi(x)\cdot\hat{n}_x=\frac{d}{dt}|_{t=d(x,y)}\ti\psi(c(t))=\sum_{\gc\in\gS_x}f'(d_\gc(x))\frac{d}{dt}_{t=d(x,y)}d_\gc(c(t))<0
$$ 
since $\gc\in\gS_x$ implies $d_\gc(x)<\gep$ and $f$ has negative derivative on $(0,\gep)$.

\medskip

In the next two cases $N_x$ and $Z_x$ are nontrivial hence infinite, being torsion free.

\medskip

{\bf Case 2:} Suppose now that $Z_x$ contains an hyperbolic element $\gc_0$ and let $A$ be the axis of $\gc_0$. It follows that all elements of $N_x$ preserve $A$ and hence attain their minimal displacement on $A$. Thus $A=\cap\min_{\gc\in N_x}(\gc)$ and since $N_x$ is normal it follows that $A$ is also $\gD_x$ invariant, and hence all elements in $\gD_x$ attain their minimal displacement on $A$. Let $y=\pi_A(x)$ be the nearest point to $x$ in $A$, let $c:[0,\infty)\to X$ be the ray from $y$ through $x$ and let $\hat{n}_x$ be the tangent to $c$ at $x$. Since $A$ is convex and $\gc$-invariant $d_\gc(x)\ge d_\gc(y),~\forall\gc\in\gS_x$. Moreover it also follow from convexity and the fact that $\BH^n$ admits no constant-distance geodesics that for every $\gc\in\gS_x$ we have $d_\gc(x)>d_\gc(y)$. Thus one can proceed arguing as in case 1.

\medskip

{\bf Case 3:} We are left with the case that $Z_x$ contains a parabolic element $\gc_0$. Since $Z_x$ is characteristic in $N_x$ and $N_x$ is normal, also $Z_x$ is normal in $\gD_x$. Moreover, since $N_x$ is of finite index in $\gD_x$ the element $\gc_0$ has only finitely many conjugates in $\gD_x$ and they are all in $Z_x$. Denote these elements by $\gc_0,\gc_1,\ldots,\gc_k$. All of them are parabolics and since they commute with each other, for every $t$ the corresponding sublevel sets intersect nontrivially:
$$
 B_t=\cap_{i=0}^k\{p\in X:d_{\gc_i}(p)\le t\}\ne\emptyset.
$$
Taking $t<\min\{d_{\gc_i}(x):i=0,\ldots,k\}$ we get a nonempty $\gD_x$-invariant closed convex set $B_t$ not containing $x$. Taking $y=\pi_{B_t}(x)$ and proceeding as in the previous cases allows us to complete the proof.
\end{proof}

\medskip

By the finiteness of the volume of $M$ we deduce:

\begin{lem}
The map $\psi$ is proper, i.e. $\psi^{-1}([0,a])$ is compact for every $a\in\BR^{\ge 0}$.
\end{lem}

\begin{proof}
Suppose $a>0$. If $\psi (x)\le a$ and $\ti x\in X$ is a lift of $x$ then for every $\gc\in\gC\setminus\{1\}$ we have $f(d_\gc(x))\le a$ and hence $d_\gc(x)\ge f^{-1}(a)$. It follows that the injectivity radius of $M$ at $x$ is at least $f^{-1}(a)/2$. Thus $\psi^{-1}([0,a])$ is contained in the $f^{-1}(a)/2$-thick part of $M$. Since $M$ has finite volume the last set is compact.
\end{proof}

Recall the following basic lemma from Morse theory \cite[Theorem 3.1]{Mi}:

\begin{lem}[Morse lemma]
Let $Q$ be a smooth manifold and $\phi:Q\to\BR^{\ge 0}$ a smooth proper map. If $\nabla\phi\ne 0$ on $\phi^{-1}(a,b)$ for some $0\le a\le b\le\infty$ then $\phi^{-1}([0,a])$ is a deformation retract of $\phi^{-1}([0,b])$.
\end{lem}

Applying the lemma to $Q=Y$ and $\phi=\psi$ we deduce that $\psi^{-1}(0)$ is a deformation retract of $Y$. 
It follows that $\pi_1(Y)\cong\pi_1(\psi^{-1}(0))$. Note that since $\gC$ acts freely on the connected manifold $\ti Y$ and $Y=\gC\backslash\ti Y$ it follows that $\gC$ is a quotient of $\pi_1(Y)$. Hence the theorem will follow if we show:

\begin{lem}
$\pi_1(\psi^{-1}(0))$ is generated by $C\cdot\vol(M)$ elements for some appropriate constant $C=C(\BH^n)$.
\end{lem}

Observe that $\psi(x)=0$ implies that $d_\gc(\ti x)\ge\gep$ for every $\gc\in\gC\setminus\{1\}$ and hence $\text{InjRad}_M(x)\ge\gep/2$.

\begin{proof}
Let $\mathcal{F}$ be a maximal $\gep/2$ discrete subset of $\psi^{-1}(0)$. Since the $\gep/4$ balls centered at $\mathcal{F}$ are disjoint and injected 
$$
 |\mathcal{F}|\le\text{Const}\cdot \vol(M).
$$
Let $U$ be the union of the $\gep/2$ balls centered at $\mathcal{F}$. Then $\psi^{-1}(0)\subset U\subset Y$ and since $\psi^{-1}(0)$ is a deformation retract of $Y$ we see that $\pi_1(\psi^{-1}(0))$ is a quotient of $\pi_1(U)$. Finally $U$ is homotopic to the simplicial complex corresponding to the nerve of the cover 
$$
 \{ B(f,\gep/2):f\in\mathcal{F}\}
$$ 
and the complexity of the letter is bounded by a constant times $|\mathcal{F}|$.
\end{proof}
This finishes the proof of the theorem.
\qed

The argument above applies for general rank one symmetric spaces. 
The proof of the Theorem \ref{thm:d(Gamma)} for higher rank spaces is of similar nature, but technically more complicated.

As an immediate consequence we deduce the following result which was originally proved as a combination of various works, notably in the work of Garland and Raghunathan \cite{Ga-Ra} on the rank one case and the work of Kazhdan \cite{Kazhdan} on the higher rank case:

\begin{cor}
Every lattice in $G$ is finitely generated. 
\end{cor}

Combining a result of Auslander \cite[8.24]{Rag} (see also \cite[Section 9]{toti}) and Theorem \ref{mostow}, one can generalize the last corollary to the case that $G$ is an arbitrary connected Lie groups.

As another application we deduce:

\begin{thm}[Kazhdan--Margulis \cite{Ka-Ma}]\label{thm:KaMa}
Given a connected semisimple Lie group $G$, there is $v_0>0$ such that the co-volume of every lattice $\gC\le_L G$ satisfies $\vol(G/\gC)\ge v_0$.
\end{thm}

\begin{proof}
Indeed, $d(\gC)\ge 1$ and hence $\vol(G/\gC)\ge\frac{1}{C}$.
\end{proof}
 
It can be shown that the minimal co-volume $v_0$ is attained, but in general it is very hard to obtain a good estimate of it. 

\begin{rem}[Kazhdan--Margulis theorem]\label{rem:KM}
The original proof of Kazhdan and Margulis \cite{Ka-Ma} shows the stronger statement that $G$ admits an identity neighborhood $U$ such that every lattice in $G$ admits a conjugate which intersects $U$ trivially. This strong version can also be deduced from our argument above. (Indeed $\psi(0)$ is nonempty, being a defamation retract of $Y$.)
In the next lecture we will make use of this stronger statement.
\end{rem}

\begin{exercise}
Show that the fact that $\psi$ attains the value $0$ implies the strong version of the Kazhdan--Margulis theorem.
\end{exercise}


\lecture{Rigidity and applications} 

In this last talk we will present (without proofs) four classical rigidity theorems and derive from each some applications.
In order to state the results in the generality of semisimple rather than simple groups we should define the notion of irreducibility.

\begin{defn}
Let $G$ be a center free semisimple Lie group without compact factors. A lattice $\gC\le_L G$ is called {\it reducible} if $G$ can be decomposed non-trivially $G=G_1\times G_2$ in a way that $(\gC\cap G_1)\times(\gC\cap G_2)$ is of finite index in $\gC$.
If no such decomposition exists, $\gC$ is said to be {\it irreducible}.
\end{defn}

A consequence of the Borel density theorem is that a lattice $\gC\le_L G$ in a center-free semisimple group with more than one simple factors is irreducible iff its projection to every proper factor is dense and iff its intersection with every proper factor is trivial \cite[Corollary 5.21]{Rag}.  

\section{Local rigidity}

Let $\gC$ be a finitely generated group and $G$ a topological group. By $\text{Hom}(\gC,G)$ we denote the space of homomorphisms $\gC\to G$ with the point-wise convergence topology. If $\gC$ is generated by $\{\gs_1,\ldots,\gs_n\}$ we can identify $\text{Hom}(\gC,G)$ with the subset 
$$
 \{(g)_1^n\in G^n:W(g_1,\ldots,g_n)=1, \forall~\text{word}~W\in F_n~\text{s.t.} ~W(\gs_1,\ldots,\gs_n)=1\}
$$
with the topology induced from $G^n$.

$G$ acts on $\text{Hom}(\gC,G)$ by conjugation, where $f^g(\cdot):=gf(\cdot)g^{-1}$.
A map $f\in\text{Hom}(\gC,G)$ is said to be locally rigid if the conjugacy class $f^G$ contains a neighborhood of $f$, i.e. if every $h\in\text{Hom}(\gC,G)$ sufficiently close to $f$ is of the form $f^g$ for some $g\in G$. A subgroup $\gC\le G$ is said to be locally rigid if the inclusion map $\gC\hookrightarrow G$ is locally rigid.

\begin{thm}[Local rigidity] (Calabi, Selberg, Weil, Margulis \cite[Ch. $VII$]{Rag}, \cite[Ch. $VII$]{Mar1})\label{thm:lr} 
Let $G$ be a connected semisimple Lie group group not locally isomorphic to $\PSL_2(\BR)$ or $\PSL_2(\BC)$. Then every irreducible lattice is locally rigid. If $G$ is locally isomorphic to $\PSL_2(\BC)$ and $\gC\le_L G$ then $\gC$ is locally rigid iff it is uniform.
\end{thm}
 
One consequence of local rigidity, which was observed by  
A. Selberg who proved local rigidity for uniform lattices in $\SL_n(\BR),~n\ge 3$, shows an interesting connection between lattices and algebraic number theorey:

\begin{prop}\label{Prop:alg}
Suppose that $\BG$ is a $\BQ$-algebraic group, $G=\BG(\BR)$ and $\gC$ a locally rigid finitely generated subgroup of $G$. Then there is a number field $\BK\subset\BR$ and $g\in G$ such that $\gC^g\le\BG(\BK)$. In particular every element in $\Ad(\gC)$ has algebraic eigenvalues.
\end{prop}

Let us explain the proof of Proposition \ref{Prop:alg}. 
It is easy to see that the deformation space $\text{Hom}(\gC,G)$ inherits a structure of a $\BQ$-algebraic variety.
Here is a useful fact about density of algebraic points in such varieties:

\begin{lem}\label{lem:X(K)-dense}
Let $X$ an algebraic variety defined over $\BQ$. Denote by $\ti\BQ$ the algebraic closure of $\BQ$ in $\BR$, i.e. the field of
all elements in $\BR$ which are algebraic over $\BQ$. Then $X(\ti\BQ)$ is dense in $X(\BR)$ with respect to the Hausdorff topology.
\end{lem}

This Lemma can be deduced from the implicit function theorem (see \cite[Lemma 3.2]{UT} for a proof given in the general setup of arbitrary global and local fields).
 
Applying the last lemma to $X=\text{Hom}(\gC,G)$ we deduce, in particular, that the inclusion $\gC\hookrightarrow G$ is a limit of ``algebraic points", which are representations whose image lie in $G(\ti\BQ)$. Let $\rho:\gC\to G$ be such an ``algebraic" representation. Assuming that $\rho$ is sufficiently close to the inclusion, we deduce from the local rigidity of $\gC$ that $\rho$ is given by conjugation by some $g\in G$. Given a finite generating set $\gS$ for $\gC$ there is a number field $\BK$ such that 
$\rho(\gs)$ is contained in $G(\BK)$ for every $\gs\in\gS$. This however implies that $\rho(\gc)\in\ G(\BK)$ for every $\gc\in\gC$. Thus we proved that $\gC$ is conjugated to a matrix group whose entries generate a number field. 
 
It follows that a lattice in semisimple linear algebraic groups without compact or $3$ dimensional factors can be represented as a matrix group with entries in some number field. This fact remains true also for non-uniform lattices in $\SL_2(\BC)$ although they are not locally rigid --- this is because these lattices are ``relatively locally rigid" inside the smaller variety of representations which send unipotents to unipotents.
It follows, for instance, that the matrix 
\[
\left(\begin{array}{cc}
\pi & 0\\
0 & {1\over \pi}
\end{array}\right)
\]
is not contained in any lattice in $\SL_2(\BC)$, since its eigenvalues are not algebraic integers.

 \section{Wang's finiteness theorem} 
 Another important application of local rigidity is Wang's finiteness theorem:

\begin{thm}[\cite{Wa}]\label{thm:Wang}
Let $G$ be a connected semisimple Lie group without compact factors, not locally isomorphic to $\SL_2(\BR)$ and $\SL_2(\BC)$. Then for every $v>0$ there are only finitely many conjugacy classes of irreducible lattices $\gC\le G$ with $\vol(G/\gC)< v$. 
\end{thm}

The proof that Wang gave to this theorem relies on four ingridient
\begin{itemize}
\item Kazhdan--Margulis theorem (Theorem \ref{thm:KaMa} and Remark \ref{rem:KM} from Lecture 3),
\item Lattices are finitely generated (Theorem \ref{thm:d(Gamma)} from Lecture 3),
\item Local rigidity theorem (Theorem \ref{thm:lr} above), and
\item The Mahaler--Chabauty compactness criterion.
\end{itemize}

As we have already discussed the first three, let us say few words about the last one:

\begin{defn}[The space of closed subgroups]\label{defn:sub_G}
Let $G$ be a locally compact group. We denote by $\text{Sub}_G$ the space of all closed subgroups of $G$ equipped with the topology, determined by the sub-base consisting of the following two types of sets:
\begin{itemize}
\item For every compact subset $K\subset G$, $O_1(K):=\{ H\in\text{Sub}_G: H\cap K=\emptyset\}$.
\item For every open subset $U\subset G$, $O_2(U):=\{ H\in\text{Sub}_G: H\cap U\ne\emptyset\}$.
\end{itemize}
\end{defn}

\begin{exercise}
Let $H_n,H\in\text{Sub}_G,~n\in\BN$. Show that $H_n\to H$ iff 
\begin{itemize}
\item $\forall h\in H,~ \exists h_n\in H_n$ such that $h_n\to h$, and
\item For any increasing sequence of integers $n_k$ and elements $h_{n_k}\in H_{n_k}$ for which $\lim h_{n_k}$ exists, we have $\lim h_{n_k}\in H$.
\end{itemize}
\end{exercise}

\begin{exercise}
If $G$ is a locally compact second countable group then
$\text{Sub}_G$ is a compact space.
\end{exercise}

\begin{thm}(Compactness criterion \cite[Theorem 1.20]{Rag})\label{thm:comp-cr}
Let $G$ be a locally compact second countable group, $\gO\subset G$ an open set and $v>0$. The set
$$
 \{\gC\le_L G:\vol(G/\gC)\le v~\text{and}~\gC\cap\gO=\{1\}\}
$$
is compact in $\text{Sub}_G$.
\end{thm}

\begin{exercise}
Prove Theorem \ref{thm:comp-cr}. 
Hint: 
$$
 \vol(G/\gC)=\sup\{\mu(K):K\subset G~\text{is compact and}~K^{-1}K\cap\gC=\emptyset\}.
$$
\end{exercise}

Let us now explain the proof of Theorem \ref{thm:Wang}. Given $G$ as in the theorem and $v>0$, we wish to show that there are only finitely many conjugacy classes of lattices $\gC\le_LG$ with $\vol(G/\gC)\le v$. By Kazhdan--Margulis' theorem there is an open identity neighborhood $\gO\subset G$ such that every lattice admits a conjugate which intersects $\gO$ trivially (see Remark \ref{rem:KM} of the previous lecture). Thus it is enough to prove that the set 
$$
 \{\gC\le_L G:\vol(G/\gC)\le v~\text{and}~\gC\cap\gO=\{1\}\}
$$
is contained in finitely many conjugacy classes. As we have just noted, this set is compact in $\text{Sub}_G$ and thus it is enough to show that every conjugacy class intersect it in an open set. This follows from local rigidity and finite generation by the following:

\begin{exercise}
Let $G$ be as above.

$(a)$ Let $\gC\le G$ be finitely presented discrete subgroup. Show that $\gC$ admits a neighborhood $U$ in $\text{Sub}_G$ such that every $\gD\in U$ is discrete and there is a homomorphism from $\gC$ to $\gD$ which is close to the inclusion in the topology of $\text{Hom}(\gC,G)$. (Hint: Map each generator to a closest point.)

$(b)$ Deduce that if $\gC$ is locally rigid then $\gD$ contains a conjugate of $\gC$, and so if $\gC$ is a lattice then, up to shrinking $U$ to a smaller neighborhood, every $\gD\in U$ is actually conjugated to $\gC$.

$(c)$ Using Hilbert's basis theorem, show that it is enough to assume in $(a)$, and hence also in $(b)$, that $\gC$ is only finitely generated rather than finitely presented.
\end{exercise}

\begin{rem}
$(1)$ In \cite{Wa} Wang proved Theorem \ref{thm:Wang} only for semisimple groups without $\PSL_2(\BR)$ and $\PSL_2(\BC)$ quotients.
For a proof of Wang's theorem for general semisimple groups (including the cases with $\PSL_2$ quotients), see \cite[Section 13.4]{Ge1}.

\end{rem}


\section{Mostow's rigidity theorem}

One of the most remarkable results concerning lattices is Mostow's strong rigidity theorem:

\begin{thm}[\cite{Mos,Pr1}]
Let $G$ be a center free semisimple Lie group without compact factors, and suppose that $G\not\cong\PSL_2(\BR)$. Let $\gC_1,\gC_2\le_L G$ be irreducible lattices. Then every isomorphism between $\gC_1$ and $\gC_2$ extends to an authomorphism of $G$.
\end{thm}

While Mostow's rigidity theorem has many applications in group theory and geometry, we will only demonstrate how it can be used, in combination with results from the previous lectures, to give an alternative proof as well as a quantitive version to Wang's finiteness theorem. 

\begin{thm}[\cite{BGLM,Ge1,Ge2}]\label{thm:QW}
Let $G$ be a semisimple Lie group without compact factors not locally isomorphic to $\SL_2(\BR),\SL_2(\BC)$. Then the number of conjugacy classes of irreducible lattices in $G$ of covolume $\le v$ is at most $v^{bv}$ where $b$ is some constant depending on $G$.
\end{thm}

It is simpler to explain that a function of that type bounds the number of classes of {\it torsion free} lattices. This weaker statement is equivalently formulated as follows: ``the number of non-isometric irreducible complete $X$-manifolds of volume $\le v$ is at most $v^{bv}$", where 
$X=G/K$ is the corresponding symmetric space. By Mostow rigidity, two $X$-manifolds $M_1$ and $M_2$ are isometric if and only if $\pi_1(M_1)\cong\pi_1(M_2)$. By Theorem \ref{thm:presentation} (see also Remark \ref{rem:bounded-presentation}(2)) the fundamental group of an $X$-manifold of volume $\le v$ admits a presentation with $\le cv$ generators and $\le cv$ length-$3$ relations. It is easy to show that the number of groups admitting such a presentation is $\le v^{bv}$ for an appropriate constant $b$.
For a proof of Theorem \ref{thm:QW} in the general case (involving torsion), see \cite{Ge2}. 

\begin{exercise}
Prove that the number $N(c,v)$ of groups $\gC$ admitting a presentation $\gC=\langle \gS | R\rangle$ with $|\gS|,|R|\le c\cdot v$ and all $r\in R$ are of length $|r|\le 3$, satisfies
$$
 v^{av}\le N(c,v)\le v^{bv}
$$
for some constants $a,b$, when $v$ is sufficiently large.
\end{exercise}

\begin{rem}
$(a)$
While the finiteness in general fails for $G\cong\SL_2(\BR),\SL_2(\BC)$, Borel proved an analog finiteness theorem for {\it arithmetic} lattices in these cases as well \cite{Bor1}. A quite precise quantitive version of Borel's theorem is given in \cite{BGLS}.

$(b)$ 
For hyperbolic manifolds of a given dimension $n\ge 4$, it is shown in \cite{BGLM} that the growth is also bounded from below by a function of the form $v^{av}$. However, for higher rank symmetric spaces, it is expected that the growth is much smaller. Assuming the Congruence Subgroup Property (which is known in many cases) and the Generalized Riemann Hypothesis, much better estimates were established in \cite{BL}. 
\end{rem}

\section{Superrigidity and Arithmeticity} 
Perhaps the most spectacular rigidity theorem is:

\begin{thm}[Margulis super-rigidity theorem \cite{Mar1}]
Let $G$ be a semisimple Lie group without compact factors and suppose $\text{rank}_\BR(G)\ge 2$. Let $\gC\le_L G$ be an irreducible lattice. Let $\BH$ be a center free simple algebraic group defined over a local field $k$ and let $\rho:\gC\to \BH(k)$ be a Zariski dense unbounded representation. Then $\rho$ extends uniquely to a representation of $G$.
\end{thm}

Let us explain briefly how super-rigidity implies arithmeticity (see Theorem \ref{thm:arith} of Lecture 1). Suppose for simplicity that $G$ is a {\it simple} group (of rank $\ge 2$). Let $\gC\le_LG$ be a lattice in $G$. By Proposition \ref{Prop:alg} we may suppose $G=\BG(\BR)$ for some linear $\BQ$-algebraic group and that $\gC\le\BG(\BK)$ for some finite extension $\BK/\BQ$. As $\gC$ is finitely generated (c.f.  Theorem \ref{thm:d(Gamma)} of Lecture 3) there is a finite set of primes $S$ such that $\gC\le\BG(\OO_\BK(S))$ where $\OO_\BK(S)$ denotes the ring of $S$-integers in $\BK$. For any valuation $\nu$ of $\BK$ let $k_\nu$ denote the corresponding completion, and consider the diagonal embedding $\BG(\BK)\hookrightarrow\prod_\nu\BG(k_\nu)$. For every finite $\nu\notin S$, the image of $\gC$ in $\BG(k_\nu)$ lies in $\BG(\OO_{k_\nu})$ where $\OO_{k_\nu}$ is the ring of integers in $k_\nu$, while for $\nu\in S$ it follows from superrigidity that the image of $\gC$ in $\BG(k_\nu)$ is bounded, as $G$ cannot be mapped non-trivially into a totally disconnected group. Since $S$ is finite, we may replace $\gC$ by a finite index subgroup $\gC'$ whose image lies in $\BG(\OO_{k_\nu})$ for any finite valuation $\nu$. It follows that $\gC'\le\BG(\OO_\BK)$. Now consider the infinite valuations. We know that the original embedding $\BK\hookrightarrow\BR$ induces the original imbedding $\gC'\hookrightarrow \BG(\BR)=G$. We claim that for every other imbedding $\BK\hookrightarrow k_\nu$ (where $k_\nu=\BR$ or $\BC$) the resulting group $\BG(k_\nu)$ is compact. Indeed, if $\BG(k_\nu)$ is non-compact for some $k_\nu$, the supperigidity theorem implies that the imbedding $\gC'\hookrightarrow \BG(k_\nu)$ induces an isomorphism $G\to\BG(k_\nu)$. This in turn implies that the embedding $\BK\hookrightarrow k_\nu$ extends to a fields isomorphism $\BR\to k_\nu$ in contrary to the assumption that $\nu$ is not the original valuation.
Denoting by $H$ the product corresponding to infinite valuations
$$
 H=\prod_{\nu~\text{infinite}}\BG(k_\nu)
$$
we get that 
\begin{itemize}
\item $H=\BH(\BR)$ where $\BH=\mathcal{R}_{\BK/\BQ}\BG$, 
\item $\gC'\le H(\BZ)$ and is of finite index there being a lattice, and 
\item the quotient map $f:H\twoheadrightarrow G$ (the projection to the ``first" factor) has compact kernel.
\end{itemize}
Thus, we have seen that $\gC$ is commensurable\footnote{Two groups are said to be {\it commensurable} if their intersection has finite index in both.} to $f(\BH(\BZ))$.
Finally, it is not hard to show that $\gC$ can be conjugated into $f(\BH(\BZ))$, by an element of the image of $f(\BH(\BQ))$. 


\section{Invariant Random Subgroups and the Nevo--Stuck--Zimmer theorem}
Let me end these lectures series by pointing out a new approach in the theory of lattices. The idea is to associate lattices with measures defined on the space of closed subgroups and to study the space of such measures. Remarkably, this naive approach has proven very profitable and was a key to various recent achievements. 

Let $G$ be a locally compact second countable group, and recall the compact space of closed subgroups $\text{Sub}_G$ with the topology defined in \ref{defn:sub_G}. $G$ acts continuously on $\text{Sub}_G$ by conjugations.
An Invariant Random Subgroup (shortly IRS) of $G$ is a Borel regular $G$-invariant probability measure on $\text{Sub}_G$. 

For any measure preserving action of $G$ on a probability space $\gO$, it can be shown that almost every stabilizer is a closed subgroup in $G$, and hence the push forward of the measure from $\gO$ to $\text{Sub}_G$ is an IRS of $G$. It can also be shown (see \cite[Theorem 2.4]{Samurais}) that every IRS in $G$ arises in this way. In particular, one can consider (the conjugacy class of) a lattice $\gC\le_L G$ as an example of an IRS --- we shall denote by $\mu_\gC$ the IRS on $G$ induced by the $G$ action on $G/\gC$ with the normalised measure.

Various people have recently become aware of the importance of IRS's in many branches of group theory, dynamics, geometry and representation theory, and there has been a lot of works studying different aspects of IRS in different context during the last three years. Here I will restrict to the work \cite{Samurais} which makes use of the notion of IRS in order to study the asymptotic of $L_2$-invariant of lattices in semi-simple Lie groups, and report few results from this work. For simplicity of the formulations of the results below let us restrict again to the case where $G$ is simple.

Some results about lattices can be extended to statement about IRS's. For instance the Borel density theorem can be generalized as follows:

\begin{thm}(\cite[Theorem 2.5]{Samurais})
Let $G$ be a simple real algebraic group and let $\mu$ be an IRS without atoms\footnote{As $G$ is simple the atoms can only be supported on the trivial normal subgroups $\{1\},G$.}. Then $\mu$ is supported on discrete and Zariski dense subgroups.
\end{thm} 

Of significant importance  in this approach is the rigidity theorem of Nevo, Stuck and Zimmer (proven in \cite{SZ} relying on the later work \cite{NZ}):

\begin{thm}\label{thm:NSZ}
Let $G$ be a simple Lie group of real rank $\ge 2$. Then every non-transitive, ergodic
probability measure preserving $G$-action is essentially free.
\end{thm}

Relying on Theorem \ref{thm:NSZ} and on property $(T)$ it is shown in \cite{Samurais}:  

\begin{thm}(\cite[Section 4]{Samurais})\label{thm:NSZ7}
Let $G$ be a noncompact simple Lie group of rank $\ge 2$. The non-atomic ergodic IRS in $G$ are precisely $\mu_\gC,~\gC\le_LG$, and the only accomulation point of the set $\{\mu_\gC:\gC\le_LG\}$ is the Dirac measure on the trivial group $\{1\}$.
\end{thm}

The following geometric result is a consequence of Theorem \ref{thm:NSZ7}:

\begin{thm}(\cite[Corolarry 4.10]{Samurais})
Let $G$ be as in the previous theorem and let $X=G/K$ be the associated symmetric space. Let $\gC_n$ be a sequence of representatives for the distinct conjugacy classes of lattices in $G$
 and let $M_n=\gC_n\backslash X$ be the corresponding $X$-orbifolds. Then for every $R>0$ we have
$$
 \lim_{n\to\infty}\frac{\vol(\{p\in M_n:\text{InjRad}_{M_n}(p)\ge R\}}{\vol (M_n)}=1.
$$
\end{thm}

Associating a finite volume manifold together with a random point in it with a probability measure on the space of pointed metric spaces, the last result is interpreted as follows: If $\rank (X)\ge 2$, every sequence of $X$-manifolds, of finite volume tending to infinity, locally converges (in the probabilistic sense of Benjamini and Schramm, see \cite{Samurais} for a precise definition) to the universal cover $X$.
The local convergence to the universal cover implies convergence of certain topological and representation theoretical invariants.
When restricting to the subsequence $\gC_{n_k}$ of uniform torsion free lattices (for which the $M_{n_k}$ are compact manifolds) this result is used to study the asymptotic of $L_2$-invariants of $G/\gC_{n_k}$ and of $M_{n_k}=\gC_{n_k}\backslash X$. In particular a uniform version of the de-George--Wallach theorem \cite{DW} about multiplicity of unitary representations (\cite[Section 7]{Samurais}) and a uniform version of the Lueck approximation theorem (\cite[Section 8]{Samurais}) are proved.

A family of lattices is called {\it uniformly discrete} if the minimal injectivity radius of the corresponding locally symmetric manifolds is uniformly bounded from below. A well known conjecture of Margulis (see \cite[Page 322]{Mar1}) suggests that the family of all torsion free arithmetic uniform lattices in a every given semisimple Lie group is uniformly discrete. Two of the main results of \cite{Samurais} are:

\begin{thm}
Let $G$ be as above, and suppose that $\gC_n\le_L G$ are non-conjugate torsion free uniformly discrete lattices. Let $\pi$ be a unitary representation of $G$ and let $m(\pi,\gC)$ be the multiplicity of $\pi$ in $L_2(G/\gC)$. Then 
$$
 \frac{m(\pi,\gC)}{\vol(G/\gC)}\to d(\pi)
$$
where $d(\pi)$ is the formal degree of $\pi$ and is nonzero iff $\pi$ is a discrete series representation.
\end{thm}

\begin{thm}
Let $G$ and $\gC_n$ be as above and denote $M_n=\gC_n\backslash X$. Then for every $k\le\dim(X)$ we have
$$
 \frac{b_k(M_n)}{\vol(M_n)}\to\gb_k(X)
$$
where $b_k$ denotes the $k$'th betti number and 
$$
 \beta_k (X) = \begin {cases} \frac{\chi (X^d)}{\vol (X^d)} & \delta (G) =0 \text { and } k = \frac12 \dim X \\ 0 & \text {otherwise},\end {cases}
 $$
where $X ^d $ is the compact dual of $X$ equipped (like $X$) with the Riemannian metric induced by the Killing form on $\mathrm{Lie}(G)$ and $\delta (G) = \mbox{rank}_\BC(G) - \mbox{rank}_\BC(K)$.
\end{thm}

\end{document}